\documentclass{article}
\usepackage{amsmath}
\usepackage{amsthm}
\usepackage{amssymb}
\usepackage{booktabs}
\usepackage{tikz}
\usepackage{enumitem}
\usepackage{stmaryrd}
\usepackage[labelformat=simple]{subcaption}
\usepackage[colorlinks=true,citecolor=black,linkcolor=black,urlcolor=blue]{hyperref}

\theoremstyle{plain}
\newtheorem{theorem}{Theorem}
\newtheorem{lemma}[theorem]{Lemma}
\newtheorem{corollary}[theorem]{Corollary}

\newtheorem{observation}[theorem]{Observation}

\theoremstyle{definition}

\newtheorem{conjecture}[theorem]{Conjecture}

\newtheorem{question}[theorem]{Question}
\theoremstyle{remark}
\newtheorem{remark}[theorem]{Remark}

\newenvironment{subproof}[1][Proof]%
{\begin{proof}[#1]}{\end{proof}}
\newcommand{\order}[1]{\ensuremath{\left\lvert#1\right\rvert}}
\newcommand{\mob}{M\"{o}bius }
\newcommand{\mobfn}[2]{\mu[#1,#2]}
\newcommand{\mobp}[1]{\mu[#1]} 
\newcommand{\ex}[1]{\overline{#1}}

\newcommand{\oneplus}{1 \oplus} 
\newcommand{\plusone}{\oplus 1}
\newcommand{\oneminus}{1 \ominus} 
\newcommand{\minusone}{\ominus 1}
\newcommand{\nsums}[2]{\oplus^{#1}#2}
\newcommand{\sumra}{\nsums{r}{\alpha}}
\newcommand{\chain}[1]{\mathcal{#1}}
\newcommand{\chainc}{\chain{C}}
\newcommand{\chainr}{\chain{R}}
\newcommand{\chaing}{\chain{G}}
\newcommand{\chainb}{\chain{B}}
\newcommand{\cprime}{c^\prime}
\newcommand{\redset}{\mathsf{R_{\pi}}}

\DeclareMathOperator{\mobmax}{max_{\mu}}
\newcommand{\balloon}{\circledcirc}
\newcommand{\ballgen}[3]{#2 \balloon_{#1} #3}

\newcommand{\ball}[2]{#1 \balloon #2}

\newcommand{\gridc}[3]{
	\foreach \i in {0,1,...,#1}{
		\draw [color=#3] ({\i+0.5}, 0.5)--({\i+0.5}, {#2+0.5});
	};
	\foreach \i in {0,1,...,#2}{
		\draw [color=#3] (0.5, {\i+0.5})--({#1+0.5}, {\i+0.5});
	};
}
\newcommand{\grid}[2]{
	\gridc{#1}{#2}{darkgray};
}
\newcommand{\sgrid}[1]{\grid{#1}{#1}}
\newcommand{\spoint}[2]{
    \filldraw (#1,#2) circle (3pt);
}    
\newcommand{\point}[2]{
	\filldraw (#1,#2) circle (4pt);
}
\newcommand{\opoint}[2]{
    \filldraw (#1,#2) circle (4pt);
    \filldraw [white] (#1,#2) circle (3pt);
}
\newcommand{\cell}[3]{
	\node at (#1,#2) {#3};
}
\newcommand{\scell}[3]{
    \cell{#1}{#2}{\small{#3}}
}

\title{2413-balloon permutations and the growth of the \mob function}
\author{David Marchant\\
    \small School of Mathematics and Statistics\\[-0.8ex]
    \small The Open University\\[-0.8ex] 
    \small Milton Keynes, MK7 6AA, UK\\
    \small\tt david.marchant@open.ac.uk}

\begin{document}
\maketitle

\begin{abstract}  
    We show that the growth of the 
    principal \mob 
    function on the permutation poset
    is at least exponential.  
    This improves on previous work, which 
    has shown that the growth 
    is at least polynomial.
       
    We define a method of constructing 
    a permutation from a smaller permutation
    which we call
    ``ballooning''.  
    We show that if $\beta$ is a 2413-balloon,
    and $\pi$ is the 2413-balloon of $\beta$,
    then
    $\mobfn{1}{\pi} = 2 \mobfn{1}{\beta}$.
    This allows us to construct a sequence
    of permutations $\pi_1, \pi_2, \pi_3\ldots$
    with lengths $n, n+4, n+8, \ldots$     
    such that $\mobfn{1}{\pi_{i+1}} = 2 \mobfn{1}{\pi_{i}}$,
    and this gives us exponential growth.
    Further, our construction method 
    gives permutations that lie within
    a hereditary class with
    finitely many simple permutations.
    
    We also find an expression for the value
    of $\mobfn{1}{\pi}$, where $\pi$ is a 2413-balloon,
    with no restriction
    on the permutation
    being ballooned.
\end{abstract}

\section{Introduction}

%
%
Let $\sigma$ and $\pi$ be permutations of natural numbers,
written in one-line notation, 
with 
$\sigma = \sigma_1 \sigma_2 \ldots \sigma_m$,
and
$\pi = \pi_1 \pi_2 \ldots \pi_n$.
We say that $\sigma$ is \emph{contained} in $\pi$ 
if there is a sequence  
$1 \leq i_1 < i_2 < \ \ldots < i_m \leq n$
such that for any $r,s \in \{1,\ldots,m\}$,
$\pi_{i_r} < \pi_{i_s}$ if and only if $\sigma_r < \sigma_s$.  
We say that $\pi$ \emph{avoids} $\sigma$ if $\pi$ does not contain $\sigma$.
The set of all permutations is a poset under the partial order given by
containment.  

%
%
A closed interval $[\sigma, \pi]$ in a poset is the sub-poset
$\{ \tau : \sigma \leq \tau \leq \pi \}$,
and a half-open interval $[\sigma, \pi)$ is the sub-poset
$\{ \tau : \sigma \leq \tau < \pi \}$.
The \mob function $\mobfn{\sigma}{\pi}$ 
is defined on an interval of a poset as follows:
for $\sigma \nleq \pi$, $\mobfn{\sigma}{\pi} = 0$;
for all $\pi$, $\mobfn{\pi}{\pi} = 1$;
and for $\sigma < \pi$, 
\[
\mobfn{\sigma}{\pi} = - \sum_{\lambda \in [\sigma, \pi)} \mobfn{\sigma}{\lambda}
\]

%
%
In this paper we are principally concerned
with the growth of the
\emph{principal \mob function}, 
$\mobp{\pi} = \mobfn{1}{\pi}$.

%
%
Applying the \mob function to 
the permutation poset
was first mentioned by
Wilf~\cite{Wilf2002}.  
Burstein, Jel{\'{i}}nek, Jel{\'{i}}nkov{\'{a}} 
and Steingr{\'{i}}msson~\cite{Burstein2011}
ask whether the principal \mob function
is unbounded, which is the first reference 
to the growth of the \mob function in the literature.
They show that $\mobp{\pi} \in \{0, \pm 1 \}$, 
and thus is bounded, 
if $\pi$ is a separable permutation,
and so is in a hereditary class with simples
$\{ 1, 12, 21\}$.
They ask (Question 27) for which classes 
is $\mobp{\pi}$ bounded?

Smith~\cite{Smith2013}
found an explicit formula for the principal \mob 
function for all permutations with a single descent.
This shows that
the growth of the \mob function
is at least quadratic.
Jel{\'{i}}nek, Kantor, Kyn{\v{c}}l and Tancer~\cite{Jelinek2020}
show how to construct a sequence of permutations
where the absolute value of the \mob function
grows according to the seventh power of the length.
In the other direction,
Brignall, Jel{\'{i}}nek, Kyn{\v{c}}l and Marchant~\cite{Brignall2020}
show that the proportion of permutations of length $n$ with principal \mob 
function equal to zero is asymptotically bounded below by 
$(1-1/e)^2 \ge 0.3995$.
%

%
%
We show that, given some permutation $\beta$,
we can construct a permutation 
that we call the ``2413-balloon'' of $\beta$.
This permutation will have four more points than $\beta$.
We then show that if 
$\pi$ is a 2413-balloon of $\beta$,
and $\beta$ is itself a 2413-balloon, 
then 
$\mobp{\pi} = 2 \mobp{\beta}$.
From this we deduce that
the growth of the principal \mob function is exponential.
If $\beta = 25314$ (which is a 2413-balloon),
then we can construct a hereditary class
that contains only the simple permutations
$\{ 1,12,21,2413,25314\}$, 
where the growth of the 
principal \mob function is exponential,
answering questions in
Burstein et al~\cite{Burstein2011}
and 
Jel{\'{i}}nek et al~\cite{Jelinek2020}.

%
%
We start by recalling some essential definitions and notation in
Section~\ref{section-definitions-and-notation},
where we also provide some extensions of 
existing results.
We formally define a 2413-balloon in 
Section~\ref{section-define-2413-balloon},
and we provide some
results which will be used
in the remainder of this paper.
In Section~\ref{section-2413-double-balloons},
we derive an expression for the 
value of $\mobp{\pi}$ when $\pi$ is a double 2413-balloon,
and following this
we show that the growth of the \mob function
is exponential in
Section~\ref{section-growth-rate-of-mu}.
We return to the topic of 2413-balloons in
Section~\ref{section-2413-balloons}, 
and derive an expression for the 
value of $\mobp{\pi}$ when $\pi$ is
any 2413-balloon.
Finally, we discuss the generalization
of the balloon operator in
Section~\ref{section-concluding-remarks}.
We also ask some questions regarding
the growth of the \mob function.

\section{Essential definitions, notation, and results}
\label{section-definitions-and-notation}

In this section we recall some standard definitions
and notation that we will use, and add some 
simple definitions and consequences of known results.  

%
%
An \emph{interval} in a permutation $\pi$ is a 
contiguous set of indexes $i, i+1, \ldots, j$
such that the set of values
$\{ \pi_i, \pi_{i+1}, \ldots, \pi_j \}$ is also contiguous.
Every permutation $\pi$ has intervals of 
length 1 and of length $\order{\pi}$,
which we call \emph{trivial intervals}.
A \emph{simple} permutation 
is a permutation that only has 
trivial intervals.
As examples, $1324$ is not simple, as, for example, 
the second and third points $(32)$
form a non-trivial interval,
whereas $2413$ is simple.

%
%
Given two permutations $\alpha$ and $\beta$,
with lengths $a$ and $b$ respectively,
the \emph{direct sum} of $\alpha$ and $\beta$,
written $\alpha \oplus \beta$ is the permutation
$
\alpha_1, \ldots, \alpha_{a},
\beta_1 + a, \ldots, \beta_{b} + a
$.
The \emph{skew sum}, $\alpha \ominus \beta$,
is the permutation
$
\alpha_1 + b, \ldots, \alpha_{a} + b,
\beta_1, \ldots, \beta_{b}
$.

%
%
Let $\alpha$ be a permutation, 
and $r$ a positive integer.
Then $\sumra$
is $\alpha \oplus \alpha \oplus \ldots \oplus \alpha \oplus \alpha$,
with $r$ occurrences of $\alpha$.

%
%
If $\pi$ is a permutation with length $n$,
then
the number of \emph{corners} of $\pi$
is the number of points of $\pi$ that are extremal in both
position and value,
that is, $\pi_1 \in \{1, n\}$ or $\pi_n \in \{1, n\}$.
It is easy to see that any permutation 
with length 2 or more can have at most two corners.
We adopt the convention that the permutation
$1$ has one corner.

%
%
If a permutation $\pi$ can be written as 
$\oneplus\oneplus\tau$,
$\oneminus\oneminus\tau$,
$\tau\plusone\plusone$, or
$\tau\minusone\minusone$,
where $\tau$ is non-empty 
(so $\order{\pi} \geq 3$),
then we say that $\pi$
has a \emph{long corner}.

%
%
We now have
\begin{lemma}
    \label{lemma-oneplus-oneplus}
    If $\pi$ has a long corner,
    then
    $\mobp{\pi} = 0$.	
\end{lemma}
\begin{lemma}
    \label{lemma-oneplus}
    If $\pi$ can be written as 
    $\pi = \oneplus\tau$,
    or
    $\pi = \tau\plusone$
    or 
    $\pi = \oneminus\tau$
    or
    $\pi = \tau\minusone$,
    and does not have a long corner,
    then
    $\mobp{\pi} = - \mobp{\tau}$.
\end{lemma}
These are well-known consequences of
Propositions~1~and~2 of
Burstein, 
Jel{\'{i}}nek, 
Jel{\'{i}}nkov{\'{a}} and 
Steingr{\'{i}}msson~\cite{Burstein2011},
and we refrain from
providing proofs here.
The reader is directed to    
Lemma 4 in~\cite{Brignall2017a}
for a proof of Lemma~\ref{lemma-oneplus-oneplus}.
Lemma~\ref{lemma-oneplus}
is a trivial extension of Corollary 3 in~\cite{Burstein2011}.  

%
%
A triple adjacency is a monotonic interval of length 3.
Smith shows that
\begin{lemma}[{%
    Smith~\cite[Lemma 1]{Smith2013}}]
    \label{lemma-triple-adjacencies}
    If a permutation $\pi$
    contains a triple adjacency then $\mobp{\pi} = 0$.
\end{lemma}
A trivial corollary to 
Lemma~\ref{lemma-triple-adjacencies} is
\begin{corollary}
    \label{corollary-monotonic-interval}
    If a permutation contains a monotonic interval
    with length 3 or more, then $\mobp{\pi} = 0$.    
\end{corollary}

%
%
A \emph{chain} in a poset interval $[1, \pi]$ is,
for our purposes,
a subset of the permutations in the interval $[1, \pi]$,
where the subset includes the elements $1$ and $\pi$,
and any two elements of the subset are comparable.
This last clause means that the subset has a total order.
If a chain $c$ has $C$ elements,
then we say that the length of $c$, written $\order{c}$,
is $C - 1$.

%
%
Philip Hall's Theorem\cite[Proposition 3.8.5]{Stanley2012} says 
that 
\[
\mobfn{\sigma}{\pi} = 
\sum_{c \in \chainc(\sigma, \pi)} (-1)^{\order{c}}  =
\sum_{i=1}^{\order{\pi} - 1} (-1)^i K_i
\]
where $\chainc(\sigma, \pi)$ is the set of chains 
in the poset interval $[\sigma, \pi]$ which  contain both $\sigma$ and $\pi$,
and $K_i$ is the number of chains of length $i$.

%
%
If $\chainc$ is a subset of the chains 
in some poset interval $[\sigma, \pi]$,
then the \emph{Hall sum} of $\chainc$ is
$\sum_{c \in \chainc} (-1)^{\order{c}}$.

%
%
A \emph{parity-reversing involution}, $\Phi: \chainc \mapsto \chainc$,
is an involution 
such that for any $c \in \chainc$,
the parities of $c$ and $\Phi(c)$
are different.

%
%
A simple corollary to Hall's Theorem is
\begin{corollary}
    \label{corollary-halls-corollary}
    If we can find a set of chains $\chainc$
    with a parity-reversing involution,
    then the Hall sum of $\chainc$ is zero.
\end{corollary}
\begin{proof}
    Because there is a parity-reversing involution, 
    the number of chains in $\chainc$ with odd length is
    equal to the number of chains with even length,
    so $\sum_{c \in \chainc} (-1)^{\order{c}}  = 0$.  	
\end{proof}

%
%
We can also use Hall's Theorem if we have a subset of chains
that meet a specific criteria:
\begin{lemma}
    \label{lemma-hall-sum-second-element-psi}
    Let $\pi$ be any permutation with length three or more.
    Let $\psi$ be a permutation with $1 < \psi < \pi$.
    Let $\chainc$ be the subset of chains in 
    the poset interval $[1, \pi]$ where the second-highest element is $\psi$.
    Then 
    \[\sum\limits_{c \in \chainc}  (-1)^{\order{c}} = - \mobp{\psi}.\]
\end{lemma}	
\begin{proof}
    If we remove $\pi$ from the chains in $\chainc$, then
    we have all of the chains in the poset interval $[1, \psi]$,
    and the Hall sum of these chains is, by definition,
    $\mobp{\psi}$.  It follows that the Hall sum
    of the chains in $\chainc$ is $ - \mobp{\psi}$.	
\end{proof}
\begin{corollary}
    \label{corollary-hall-sum-second-highest-set}
    Given a permutation $\pi$,
    and a set of permutations $S$
    where every $\sigma \in S$ satisfies
    $1 < \sigma < \pi$,
    then 
    if $\chainc$ is the set of chains in the poset interval
    $[1, \pi]$ where the second-highest element is in $S$,
    then the Hall sum of $\chainc$ is
    $- \sum_{\sigma \in S} \mobp{\sigma}$.    
\end{corollary}
\begin{proof}	
    First, partition $\chainc$ based on the second-highest element,
    and then
    apply Lemma~\ref{lemma-hall-sum-second-element-psi} to each
    partition.
\end{proof}

%
%
When discussing chains, in general we will only be interested
in a small subset of the chain containing two or three elements.
We say that a \emph{segment} of some chain $c$ 
is a non-empty subset of the elements in $c$ with the property that
any element not in the segment is 
either less than every element in the segment,
or is greater than every element in the segment.

%
%
In our proofs, given a set of chains $\chainc$,
and a chain $c \in \chainc$,
we will frequently want to construct a chain $\cprime$
by using a parity-reversing involution $\Phi$.
Strictly speaking, $\Phi$ 
is a function that maps a set of permutations
(which is a chain)
to a set of permutations
(which may not be a chain).
As examples, if $\Phi(c)$
removes the largest or smallest element of $c$,
or adds an element so that $\Phi(c)$ does not have
a total order, then $\Phi(c)$ is not a chain.
To show that $\Phi$ is a parity-reversing involution
we will need to show that $\Phi(c)$ is a chain in $\chainc$,
and that $c$ and $\Phi(c)$ have opposite parities.
In our discussions, we will typically set
$\cprime = \Phi(c)$, 
and then show that the 
set of permutations $\cprime$ is a chain.
We will then, without further comment, 
treat $\cprime$ as a chain.

\section{2413-Balloons}
\label{section-define-2413-balloon}

In this section we define the vocabulary and notation
specific to this paper.
We also present some general results which will be used
in later sections.

%
%
Given a non-empty permutation $\beta$,
the \emph{2413-balloon} of $\beta$
is the permutation formed by
inserting $\beta$ into the centre of $2413$,
which we write as $\ball{2413}{\beta}$.
Formally, we have
\begin{align*}
(\ball{2413}{\beta})_i
& =
\begin{cases}
2 & \text{if $i = 1$}\\
\order{\beta} + 4 & \text{if $i = 2$}\\
\beta_{i-2} + 2 & \text{if $i > 2$ and $i \leq \order{\beta} + 2$ }\\
1 & \text{if $i = \order{\beta} + 3$}\\
\order{\beta} + 3 & \text{if $i = \order{\beta} + 4$}\\ 
\end{cases}
\end{align*}
Figure~\ref{figure-2413-balloons-a} 
shows $\ball{2413}{\beta}$.
\begin{figure}
    \begin{center}
        \begin{subfigure}[t]{0.35\textwidth}
            \begin{center}
                \begin{tikzpicture}[scale=0.5]
                \sgrid{5};
                \point{1}{2};
                \point{2}{5};
                \scell{3}{3}{$\beta$};
                \point{4}{1};
                \point{5}{4};
                \end{tikzpicture}
            \end{center}
            \caption{}
            \label{figure-2413-balloons-a} 
        \end{subfigure}
        \begin{subfigure}[t]{0.55\textwidth}
            \begin{center}
                \begin{tikzpicture}[scale=0.5]
                \sgrid{9};
                \point{1}{2};
                \point{2}{9};
                \point{3}{4};
                \point{4}{7};
                \scell{5}{5}{$\gamma$};
                \point{6}{3};
                \point{7}{6};
                \point{8}{1};
                \point{9}{8};
                \end{tikzpicture}
            \end{center}
            \caption{}
            \label{figure-2413-balloons-b} 
        \end{subfigure}
    \end{center}
    \caption{
        (a) The 2413-balloon $\ball{2413}{\beta}$ and
        (b) the double 2413-balloon $\ball{2413}{\ball{2413}{\gamma}}$.}
    \label{figure-2413-balloons} 
\end{figure}

%
%
The balloon operation as defined has to be
right-associative and the definition given does not
support overriding right-associativity.
In other words, 
$\ball{2413}{\ball{2413}{\beta}}$
must be
$\ball{2413}{(\ball{2413}{\beta})}$,
and
$\ball{(\ball{2413}{2413})}{\beta}$ is not defined.
In 
Section~\ref{section-concluding-remarks} we  
suggest how the balloon operation could be generalized.

%
%
Given some $\pi = \ball{2413}{\beta}$,
if $\beta$ is itself a 2413-balloon,
so $\pi = \ball{2413}{\ball{2413}{\gamma}}$,
then we say that $\pi$ is a \emph{double 2413-balloon}.
Figure~\ref{figure-2413-balloons-b} shows a double 2413-balloon.

%
%
\begin{remark}
    We note that we can write $\ball{2413}{\beta}$
    as the inflation  
    $25314 [1,1,\beta,1,1]$ 
    (see Albert and Atkinson~\cite{Albert2005}
    for further details of inflations).
    In this paper we use balloon notation,
    as we feel that this leads to a simpler exposition.
\end{remark}

%
%
If we have $\pi = \ball{2413}{\beta}$,
and we have some $\sigma$ 
with $\beta \leq \sigma < \pi$,
we will frequently want to 
represent $\sigma$
in terms of 
sub-permutations of $2413$ and 
the permutation $\beta$.
We start by colouring 
the extremal points of $\pi$ red,
and all remaining points black.  
Note that the red points are a 2413 permutation,
and the black points are $\beta$.

Now consider 
a specific embedding of $\sigma$ into $\pi$,
where we use all of the black points ($\beta$).
If the embedding is 
monochromatic ($\sigma = \beta$) then
we require no special notation.
If the embedding is not monochromatic,
then it must be the case that
only some of the red points are used.
We take $2413$, and mark the red points that are 
unused with an overline, and then write
$\sigma$ using our balloon notation.
As an example of this, 
if
$\pi = \ball{2413}{21} = 264315$, and $\sigma = 213$,
then we could represent $\sigma$ as
$\ball{\ex{2}\ex{4}\ex{1}3}{21}$.
This example is shown in Figure~\ref{figure-overbar-notation}.
\begin{figure}
    \begin{center}
        \begin{tikzpicture}[scale=0.5]
            \sgrid{6}
            \opoint{1}{2}
            \opoint{2}{6}
            \point{3}{4}
            \point{4}{3}
            \opoint{5}{1}
            \point{6}{5}
        \end{tikzpicture}
    \end{center}
    \caption{An embedding of $213 = \ball{\ex{2}\ex{4}\ex{1}3}{21}$ in $264315 = \ball{2413}{21}$.}
    \label{figure-overbar-notation}
\end{figure}
We can see that if 
$\beta \leq \sigma < \ball{2413}{\beta}$, 
and $\beta$ is not monotonic (i.e., not the 
identity permutation or its reverse), then
there is a unique way to represent
$\sigma$ using this notation.

%
%
If we have $\pi = \ball{2413}{\beta}$, 
and $\sigma$ is a permutation such that
$\beta \leq \sigma < \pi$,
then we say that
$\sigma$ is a \emph{reduction} of $\pi$.
If $\sigma$ is a reduction of $\pi = \ball{2413}{\beta}$,
and there is no $\eta$ with $\order{\eta} <\order{\beta}$
such that $\sigma$ is a reduction of $\ball{2413}{\eta}$,
then we say that $\sigma$ is a
\emph{proper reduction} of $\pi$.
A reduction of $\pi$ that is not a proper reduction
is an \emph{improper reduction}.

The following case-by-case analysis shows the
improper reductions (of $\pi$)
based on the form of $\beta$.

\begin{itemize}
	\item If $\beta$ is a 2413-balloon,
	then
	$\beta$ is the only improper reduction of $\pi$.
	\item If $\beta$ is not a 2413-balloon,
	and $\beta$ has no corners,
	then
    there are no improper reductions of $\pi$.
    \item If $\beta$ has one corner,
    then there are four improper reductions of $\pi$.
    As an example, if $\beta = 1 \oplus \gamma$, then the 
    improper reductions of $\pi$ are
    $\ball{\ex{2}\ex{4}13}{\beta}$,
    $\ball{\ex{2}\ex{4}1\ex{3}}{\beta}$,
    $\ball{\ex{2}\ex{4}\ex{1}3}{\beta}$,
    and
    $\beta$.  
    \item If $\beta$ has two corners, then
    there are seven improper reductions of $\pi$.
    As an example, if $\beta = 1 \oplus \gamma \oplus 1$,
    then the improper reductions are
    $\ball{\ex{2}\ex{4}13}{\beta}$,
    $\ball{24\ex{1}\ex{3}}{\beta}$,
    $\ball{2\ex{4}\ex{1}\ex{3}}{\beta}$,
    $\ball{\ex{2}4\ex{1}\ex{3}}{\beta}$,
    $\ball{\ex{2}\ex{4}1\ex{3}}{\beta}$,
    $\ball{\ex{2}\ex{4}\ex{1}3}{\beta}$, 
    and    
    $\beta$.
\end{itemize}

The set of permutations that are 
proper reductions of $\pi$
is written as $\redset$.
Figure~\ref{figure-2413-reductions}
shows all the reductions (proper and improper)
of $\pi = \ball{2413}{\beta}$.

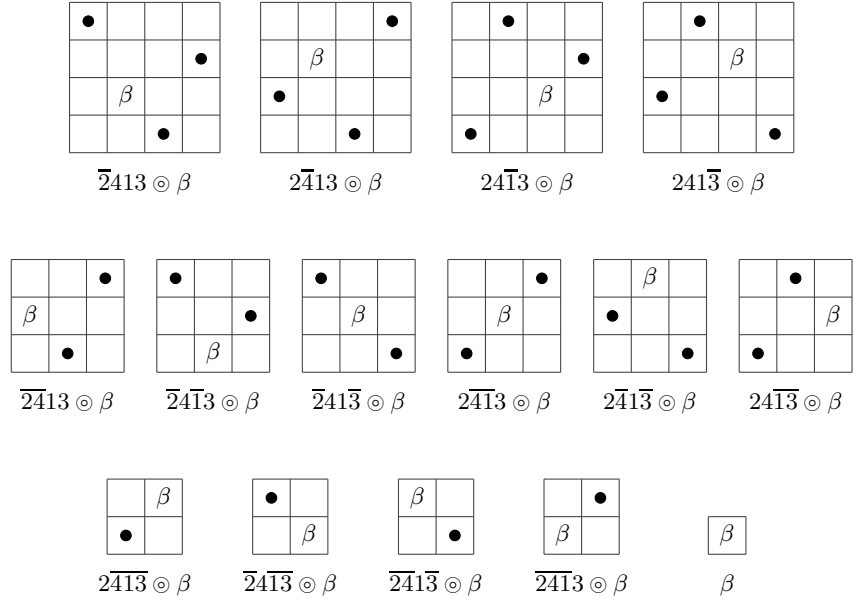
\begin{figure}
    \begin{center}
        \begin{subfigure}[t]{0.2\textwidth}
            \centering
            \begin{tikzpicture}[scale=0.5]
            \sgrid{4}
            \point{1}{4}
            \scell{2}{2}{$\beta$}
            \point{3}{1}
            \point{4}{3}
            \end{tikzpicture}
            \caption*{$\ball{\ex{2}413}{\beta}$}
        \end{subfigure}
        \begin{subfigure}[t]{0.2\textwidth}
            \centering
            \begin{tikzpicture}[scale=0.5]
            \sgrid{4}
            \point{1}{2}
            \scell{2}{3}{$\beta$}
            \point{3}{1}
            \point{4}{4}
            \end{tikzpicture}
            \caption*{$\ball{2\ex{4}13}{\beta}$}
        \end{subfigure}
        \begin{subfigure}[t]{0.2\textwidth}
            \centering
            \begin{tikzpicture}[scale=0.5]
            \sgrid{4}
            \point{1}{1}
            \point{2}{4}
            \scell{3}{2}{$\beta$}
            \point{4}{3}
            \end{tikzpicture}
            \caption*{$\ball{24\ex{1}3}{\beta}$}
        \end{subfigure}
        \begin{subfigure}[t]{0.2\textwidth}
            \centering
            \begin{tikzpicture}[scale=0.5]
            \sgrid{4}
            \point{1}{2}
            \point{2}{4}
            \scell{3}{3}{$\beta$}
            \point{4}{1}
            \end{tikzpicture}
            \caption*{$\ball{241\ex{3}}{\beta}$}
        \end{subfigure}	
        \vspace{1\baselineskip}
    \end{center} 
    \begin{center}
        \begin{subfigure}[t]{0.15\textwidth}
            \centering
            \begin{tikzpicture}[scale=0.5]
            \sgrid{3}
            \scell{1}{2}{$\beta$}
            \point{2}{1}
            \point{3}{3}
            \end{tikzpicture}
            \caption*{$\ball{\ex{2}\ex{4}13}{\beta}$}
        \end{subfigure}
        \begin{subfigure}[t]{0.15\textwidth}
            \centering
            \begin{tikzpicture}[scale=0.5]
            \sgrid{3}
            \point{1}{3}
            \scell{2}{1}{$\beta$}
            \point{3}{2}
            \end{tikzpicture}
            \caption*{$\ball{\ex{2}4\ex{1}3}{\beta}$}
        \end{subfigure}
        \begin{subfigure}[t]{0.15\textwidth}
            \centering
            \begin{tikzpicture}[scale=0.5]
            \sgrid{3}
            \point{1}{3}
            \scell{2}{2}{$\beta$}
            \point{3}{1}
            \end{tikzpicture}
            \caption*{$\ball{\ex{2}41\ex{3}}{\beta}$}
        \end{subfigure}
        \begin{subfigure}[t]{0.15\textwidth}
            \centering
            \begin{tikzpicture}[scale=0.5]
            \sgrid{3}
            \point{1}{1}
            \scell{2}{2}{$\beta$}
            \point{3}{3}
            \end{tikzpicture}
            \caption*{$\ball{2\ex{4}\ex{1}3}{\beta}$}
        \end{subfigure}	
        \begin{subfigure}[t]{0.15\textwidth}
            \centering
            \begin{tikzpicture}[scale=0.5]
            \sgrid{3}
            \point{1}{2}
            \scell{2}{3}{$\beta$}
            \point{3}{1}
            \end{tikzpicture}
            \caption*{$\ball{2\ex{4}1\ex{3}}{\beta}$}
        \end{subfigure}	
        \begin{subfigure}[t]{0.15\textwidth}
            \centering
            \begin{tikzpicture}[scale=0.5]
            \sgrid{3}
            \point{1}{1}
            \point{2}{3}
            \scell{3}{2}{$\beta$}
            \end{tikzpicture}
            \caption*{$\ball{24\ex{1}\ex{3}}{\beta}$}
        \end{subfigure}	
        \vspace{1\baselineskip}
    \end{center}
    \begin{center}
        \begin{subfigure}[t]{0.15\textwidth}
            \centering
            \begin{tikzpicture}[scale=0.5]
            \sgrid{2}
            \point{1}{1}
            \scell{2}{2}{$\beta$}
            \end{tikzpicture}
            \caption*{$\ball{2\ex{4}\ex{1}\ex{3}}{\beta}$}
        \end{subfigure}
        \begin{subfigure}[t]{0.15\textwidth}
            \centering
            \begin{tikzpicture}[scale=0.5]
            \sgrid{2}
            \point{1}{2}
            \scell{2}{1}{$\beta$}
            \end{tikzpicture}
            \caption*{$\ball{\ex{2}4\ex{1}\ex{3}}{\beta}$}
        \end{subfigure}
        \begin{subfigure}[t]{0.15\textwidth}
            \centering
            \begin{tikzpicture}[scale=0.5]
            \sgrid{2}
            \scell{1}{2}{$\beta$}
            \point{2}{1}
            \end{tikzpicture}
            \caption*{$\ball{\ex{2}\ex{4}1\ex{3}}{\beta}$}
        \end{subfigure}	
        \begin{subfigure}[t]{0.15\textwidth}
            \centering
            \begin{tikzpicture}[scale=0.5]
            \sgrid{2}
            \scell{1}{1}{$\beta$}
            \point{2}{2}
            \end{tikzpicture}
            \caption*{$\ball{\ex{2}\ex{4}\ex{1}3}{\beta}$}
        \end{subfigure}
        \begin{subfigure}[t]{0.15\textwidth}
    	    \centering
        	\begin{tikzpicture}[scale=0.5]
    	    \sgrid{1}
        	\scell{1}{1}{$\beta$}
        	\end{tikzpicture}
    	    \caption*{$\beta$}
    \end{subfigure}
    \end{center}
    \caption{Reductions of $\pi = \ball{2413}{\beta}$.  Some may not be proper reductions, depending on $\beta$.}	      
    \label{figure-2413-reductions}
\end{figure}

%
%
The strategy that we will use 
in 
Sections~\ref{section-2413-double-balloons} 
and~\ref{section-2413-balloons} is
to partition the chains in the poset interval 
$[1,\pi]$ into three sets,
$\chainr$, 
$\chaing$, and
$\chainb$. 
We then show that there are parity-reversing involutions
on the sets 
$\chaing$ and
$\chainb$,
and therefore, by 
Corollary~\ref{corollary-halls-corollary},
the Hall sum for each of these sets is zero,
and so $\mobp{\pi}$ is given by the 
Hall sum of the set $\chainr$.
Finally, we show that the Hall sum of $\chainr$
can be written in terms of $\mobp{\beta}$.

%
%
The chains in $\chainr$ are those chains
where the second-highest element 
is a proper reduction of $\pi$,
so if $\kappa_c$ is the second-highest
element of a chain $c$, then 
$c \in \chainr$ if and only if
$\kappa_c \in \redset$.
Note that, as mentioned earlier,
the members of $\redset$,
and hence the chains in $\chainr$,
depend on the form of $\pi$.
It is easy to see that for any permutation $\sigma \in \redset$
we must have $\order{\sigma} \geq \order{\beta}$.

We have some results that
are independent of $\redset$, 
and, once we have given some some further 
definitions, we present these in
the current section
to avoid repetition.

Let $\pi$ be a 2413-balloon, and let $c$ be any chain
in the poset interval $[1, \pi]$.

%
%
Since the top of the chain is, by definition, a 2413-balloon,
it follows that $c$ has a unique maximal
segment that includes the element $\pi$,
where every element in the segment is a 2413-balloon.
We call the smallest element 
in this segment 
the \emph{least 2413-balloon}\footnote{The name 
    should really be ``least 2413-balloon in the chain that has only 2413-balloons above it''.}.

%
%
Further, since the permutation 1 is not a 2413-balloon,
it follows that $c$ has an element that is immediately below
the least 2413-balloon in the chain,
and we call this element the \emph{pivot}.

%
%
We define $\phi_c$ to be the least 2413-balloon in $c$,
$\psi_c$ to be the pivot in $c$,
$\tau_c$ to be the permutation that satisfies
$\ball{2413}{\tau_c} = \phi_c$,
and $\kappa_c$ to be the second-highest element of $c$.
Note that $\phi_c$ and $\psi_c$ must be distinct,
but we can have $\tau_c = \psi_c$.
Further, $\kappa_c$ is independent, and 
may be the same as $\phi_c$, $\psi_c$ or $\tau_c$.
Figure~\ref{figure-example-chains} shows some
example chains, highlighting these elements.
%
%
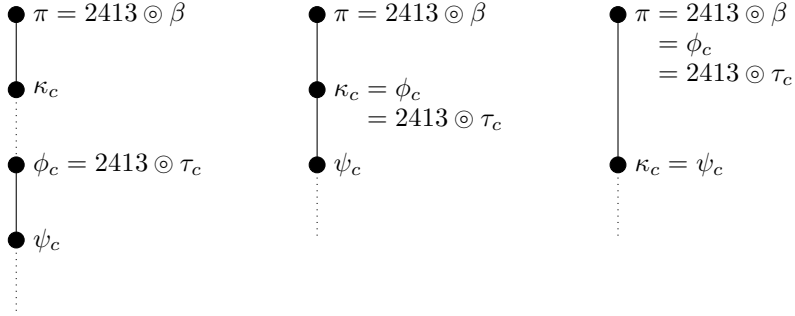
\begin{figure}
    \begin{center}
                \begin{tikzpicture}[]                
                \draw [] (0,5) -- (0,4);
                \draw [dotted] (0,4) -- (0,3);
                \draw [] (0,3) -- (0,2);
                \draw [dotted] (0,2) -- (0,1);
                \spoint{0}{5};
                \spoint{0}{4};
                \spoint{0}{3};
                \spoint{0}{2};
                \node [right] at (0.1, 5.0) {$\pi = \ball{2413}{\beta}$};
                \node [right] at (0.1, 4.0) {$\kappa_c$};
                \node [right] at (0.1, 3.0) {$\phi_c= \ball{2413}{\tau_c}$};
                \node [right] at (0.1, 2.0) {$\psi_c$};
                \draw [] (4,5) -- (4,4);
                \draw [] (4,4) -- (4,3);
                \draw [dotted] (4,3) -- (4,2);
                \spoint{4}{5};
                \spoint{4}{4};
                \spoint{4}{3};
                \node [right] at (4.1, 5.0) {$\pi = \ball{2413}{\beta}$};
                \node [right] at (4.1, 4.0) {$\kappa_c = \phi_c$};
                \node [right] at (4.1, 3.6) {$\phantom{\kappa_c} = \ball{2413}{\tau_c}$};
                \node [right] at (4.1, 3.0) {$\psi_c$};
                \draw [] (8,5) -- (8,3);
                \draw [dotted] (8,3) -- (8,2);
                \spoint{8}{5};
                \spoint{8}{3};
                \node [right] at (8.1, 5.0) {$\pi = \ball{2413}{\beta}$};
                \node [right] at (8.1, 4.6) {$\phantom{\pi} = \phi_c$};
                \node [right] at (8.1, 4.2) {$\phantom{\pi} = \ball{2413}{\tau_c}$};
                \node [right] at (8.1, 3.0) {$\kappa_c = \psi_c$};
                \end{tikzpicture}
    \end{center}
    \caption{Examples of chains, showing some possible relationships between $\pi$, $\kappa_c$, $\phi_c$, and $\psi_c$.}            
    \label{figure-example-chains}
\end{figure}    

%
%
We are now in a position to give a definition of the 
sets
$\chainr$,
$\chaing$, and
$\chainb$.
This definition
depends on the set of proper reductions of $\pi$, $\redset$,
which, as stated earlier, 
depends on the form of $\beta$.

Let $\chainc$ be the set of chains in the poset interval $[1, \pi]$.
We define subsets of $\chainc$ as follows:
\begin{align*}
\chainr &= \{ c : c \in \chainc \text{ and } \kappa_c \in \redset \}, \\
\chaing &= \{ c : c \in \chainc \setminus \chainr \text{ and } \psi_c \leq 2413 \}, \\
\chainb &= \{ c : c \in \chainc \setminus (\chainr \cup \chaing) \}.
\end{align*}
Clearly, every chain in $[1, \pi]$ is included
in exactly one of these subsets, 
and so these sets are a partition of the chains.

%
%
Given a pivot $\psi_c$, there is a unique 
permutation $\eta_c$ which we call the \emph{core}
of $\psi_c$.  
In essence, $\eta_c$ is the smallest permutation
such that $\psi_c < \ball{2413}{\eta_c}$.
To determine the core,
we use the following algorithm:
\begin{align*}
\text{If $\psi_c$ can be written as~~}
& 1 \ominus ( ( \eta \ominus 1 ) \oplus 1) 
\text{~~or~~}
( ( 1 \oplus \eta ) \ominus 1 ) \oplus 1 \\
\text{or~~}
&
1 \oplus ( 1 \ominus ( \eta \oplus 1 ) ) 
\text{~~or~~}
( 1 \oplus ( 1 \ominus \eta ) ) \ominus 1, \\
\text{then set~~}
& \eta_c = \eta.\\
\text{Otherwise, if $\psi_c$ can be written as~~}
& ( \eta \ominus 1 ) \oplus 1 
\text{~~or~~}
1 \ominus ( \eta \oplus 1 ) 
\text{~~or~~}
1 \ominus \eta \ominus 1 \\
\text{or~~}
& 1 \oplus \eta \oplus 1 
\text{~~or~~}
( 1 \oplus \eta ) \ominus 1 
\text{~~or~~}
1 \oplus ( 1 \ominus \eta ), \\
\text{then set~~}
& \eta_c = \eta.\\
\text{Otherwise, if $\psi_c$ can be written as~}
& 1 \oplus \eta 
\text{~~or~~}
1 \ominus \eta 
\text{~~or~~}
\eta \ominus 1 
\text{~~or~~}
\eta \oplus 1, \\
\text{then set~~}
& \eta_c = \eta.\\
\text{Otherwise, set~~}
& \eta_c = \psi_c.\\
\end{align*}

Since we have $\psi_c < \phi_c = \ball{2413}{\tau_c}$,
it is easy to see that 
$\eta_c \leq \tau_c$.
Note that $\ball{2413}{\eta_c}$ is the smallest 2413-balloon
that contains $\psi_c$.

%
%
We now define two functions, one for each of
$\chaing$ and
$\chainb$,
which will give us parity-reversing involutions.

\begin{align*}
\Phi_{\chaing}(c) & =
\begin{cases}
c \setminus \{2413\} & \text{If $\psi_c = 2413$} \\
c \cup \{2413\} & \text{If $\psi_c < 2413$} \\
\end{cases} \\
\Phi_{\chainb}(c) & =
\begin{cases}
c \setminus \{\ball{2413}{\eta_c}\} & \text{If $\eta_c = \tau_c$} \\
c \cup \{\ball{2413}{\eta_c}\} & \text{If $\eta_c < \tau_c$} \\
\end{cases} \\
\end{align*}

\begin{remark}
    If we were to allow the ballooning of the empty permutation $\epsilon$,
    and then treat $2413$ as $\ball{2413}{\epsilon}$
    then $\Phi_{\chaing}(c)$ is subsumed by 
    $\Phi_{\chainb}(c)$.  Doing this, however, introduces 
    additional complications in later proofs, 
    and so we prefer two involutions.
\end{remark}

For $\Phi_{\chaing}(c)$ to be a 
parity-reversing involution on $\chaing$,
we need to show that
if $c \in \chaing$, then
$\Phi_{\chaing}(c)$ is a chain, 
that $\Phi(c) \in \chaing$, 
and that
$c$ and $\Phi(c)$ have different parities.
It is easy to see that 
this last condition is true.
A similar comment applies to 
$\Phi_{\chainb}(c)$ and $\chainb$.

For $\Phi_{\chaing}(c)$ we can show that 
all the conditions hold for any $\redset$,
regardless of the form of $\beta$.
For $\Phi_{\chainb}(c)$
we show that some weaker conditions hold
for an arbitrary subset of the reductions of $\pi$,
and then, when we have an explicit set of proper reductions,
we show that all conditions hold.
The following Lemma
gives us a result that applies to 
$\Phi_{\chaing}(c)$ and $\Phi_{\chainb}(c)$
for any $\redset$,
and we
will use this result in both 
Section~\ref{section-2413-double-balloons}
and
Section~\ref{section-2413-balloons}.

%
%
\begin{lemma}
    \label{lemma-2413-balloon-phi-gb}
    Let $\pi = \ball{2413}{\beta}$, 
    with $\order{\beta} > 4$, 
    and let 
    $\chainr$,
    $\chaing$, and
    $\chainb$
    be as defined above.
    
    \begin{enumerate}[label=(\alph*)]
        \item \label{enum-lemma-balloon-g}
        If $c \in \chaing$, 
        then $\Phi_{\chaing}(c) \in \chaing$.
        
        \item \label{enum-lemma-balloon-b-eq}
        If $c \in \chainb$, 
        with $\eta_c = \tau_c$,
        and $\Phi_{\chainb}(c)$ is a chain,
        then $\Phi_{\chainb}(c) \in \chainb \cup \chainr$.
        
        \item \label{enum-lemma-balloon-b-lt}
        If $c \in \chainb$, 
        with $\eta_c < \tau_c$,
        then $\Phi_{\chainb}(c) \in \chainb \cup \chainr$.
    \end{enumerate}
\end{lemma}
\begin{proof}
    \textbf{Case~\ref{enum-lemma-balloon-g}.}
    First, assume that $c \in \chaing$ with $\psi_c = 2413$.
    Then $c$ contains a segment 
    $2413 < \ball{2413}{\tau_c}$,
    and $\cprime = \Phi_{\chaing}(c) = c \setminus \{ 2413 \}$.
    We can see that $\cprime$ is a chain,
    as 2413 is neither the smallest nor the largest entry in $\cprime$.
    Further, $\psi_{\cprime} < 2413$.  
    Since $\order{\beta} > 4$, 
    and $\order{\psi_{\cprime}} < 4$
    we must have $\cprime \not \in \chainr$,
    and therefore $\cprime \in \chaing$.
    
    Now assume that $c \in \chaing$ with $\psi_c < 2413$.
    Then $c$ contains a segment 
    $\psi_c < \ball{2413}{\tau_c}$,
    and $\cprime = \Phi_{\chaing}(c) = c \cup \{ 2413 \}$.
    We can see that $\cprime$ is a chain,
    since $\psi_c < 2413 < \ball{2413}{\tau_c}$,
    and further, $\psi_{\cprime} = 2413$.
    Since $\order{\beta} > 4$, 
    and $\order{\psi_{\cprime}} = 4$
    we must have $\cprime \not \in \chainr$,
    and therefore $\cprime \in \chaing$. 
    
    \textbf{Case~\ref{enum-lemma-balloon-b-eq}.}
    Let $c$ be a chain in $\chainb$, 
    with $\eta_c = \tau_c$.
    Then $c$ contains a segment
    $\psi_c < \ball{2413}{\tau_c}$,
    and $\cprime = \Phi_{\chainb}(c) = c \setminus \{ \ball{2413}{\tau_c} \}$.
    If $\tau_c = \beta$, then $\cprime$ is not a chain,
    so we must have $\tau_c < \beta$, and therefore
    $\cprime$ is a chain that contains a segment
    $\psi_c < \ball{2413}{\gamma}$,
    with $\tau_c < \gamma$.
    Now, $\psi_c$ is the pivot of $\cprime$,
    so we cannot have $\cprime \in \chaing$
    as this would imply that $c \in \chaing$,
    which is a contradiction.
    Thus either
    $\cprime \in \chainr$ or $\cprime \in \chainb$.
        
    \textbf{Case~\ref{enum-lemma-balloon-b-lt}.}
    Let $c$ be a chain in $\chainb$, 
    with $\eta_c < \tau_c$.
    Then $c$ contains a segment
    $\psi_c < \ball{2413}{\tau_c}$,
    and $\cprime = \Phi_{\chainb}(c) = c \cup \{ \ball{2413}{\eta_c} \}$.
    We can see that $\cprime$ is a chain since
    $\psi_c < \ball{2413}{\eta_c} < \ball{2413}{\tau_c}$.
    Now, $\psi_c$ is the pivot of $\cprime$,
    so we cannot have $\cprime \in \chaing$
    as this would imply that $c \in \chaing$,
    which is a contradiction.
    So either
    $\cprime \in \chainr$ or $\cprime \in \chainb$.
\end{proof}

%
%
We now have
\begin{observation}
    \label{observation-all-we-have-to-do}
    If $\pi = \ball{2413}{\beta}$,
    with $\order{\beta} > 4$,
    then to show that
    $\Phi_{\chainb}$ is a parity-reversing involution on $\chainb$
    it is sufficient to show that:    
    \begin{enumerate}[label=(\alph*)]
        \item \label{enum-observation-all-we-have-to-do-b-eq}
        If $c \in \chainb$ 
        and $\eta_c = \tau_c$, 
        then $\Phi_{\chainb}(c)$ is a chain, and
        $\Phi_{\chainb}(c) \not\in \chainr$.
                
        \item \label{enum-observation-all-we-have-to-do-b-chain}
        If $c \in \chainb$,
        and $\eta_c < \tau_c$, 
        then $\Phi_{\chainb}(c) \not\in \chainr$.
    \end{enumerate}

    Further,
    if
    $\Phi_{\chainb}$ is a parity-reversing involution on $\chainb$,
    then 
    $\mobp{\pi} = - \sum_{\sigma \in \redset} \mobp{\sigma}$.    
\end{observation}
\begin{proof}
    Combining~\ref{enum-observation-all-we-have-to-do-b-eq}
    and~\ref{enum-observation-all-we-have-to-do-b-chain}
    above
    with cases~\ref{enum-lemma-balloon-b-eq}
    and~\ref{enum-lemma-balloon-b-lt}
    of Lemma~\ref{lemma-2413-balloon-phi-gb}
    gives us that 
    $\Phi_{\chainb}$ is a parity-reversing involution on $\chainb$.
    
    This now gives us that
    $\sum_{c \in \chainb} (-1)^{\order{c}} = 0$.     
    By Lemma~\ref{lemma-2413-balloon-phi-gb}, we have
    $\sum_{c \in \chaing} (-1)^{\order{c}} = 0$, 
    so we must have
    $\mobp{\pi} = \sum_{c \in \chainr} (-1)^{\order{c}}$.
    Since the chains in $\chainr$ are defined by the 
    second-highest element ($\kappa_c$) being in $\redset$, 
    the final part of the observation follows
    by applying Corollary~\ref{corollary-hall-sum-second-highest-set}.
\end{proof}

\section{The \mob function of double 2413-balloons}
\label{section-2413-double-balloons}

We are now able to state and prove our first major result.
\begin{theorem}
    \label{theorem-2413-balloon-beta-is-a-balloon}
    Let $\pi = \ball{2413}{\beta}$,
    where $\beta$ is a 2413-balloon,
    Then $\mobp{\pi} = 2 \mobp{\beta}$.	
\end{theorem}
\begin{proof}
    Note that $\beta \not \in \redset$, and further
    that $\order{\beta} > 4$, 
    since $\beta$ is a 2413-balloon.
    
    Using Observation~\ref{observation-all-we-have-to-do},
    we will show that
    $\Phi_{\chainb}$ is a parity-reversing involution on $\chainb$.
    Once we have shown that we have parity-reversing involutions,
    we will then show how to express the Hall sum of $\chainr$
    in terms of $\mobp{\beta}$.
    
    %
    %
    \begin{subproof}[Proof that $\Phi_{\chainb}$ is a parity-reversing involution on $\chainb$.]
        Let $c$ be a chain in $\chainb$.
        
        %
        First, assume that $\eta_c = \tau_c$.
        If $\tau_c = \beta$, then either 
        $\psi_c$ is a proper reduction of $\pi$,
        or $\psi_c = \beta$.
        In the first case, $c \in \chainr$,
        and in the second case 
        $\psi_c$ is a 2413-balloon, 
        and these are both contradictions.
        Thus
        we must have $\tau_c < \beta$,
        and so there is at least one permutation
        in $c$ greater than $\phi_c$.  
        It follows that $\cprime$ is a chain.
        We now show that $\cprime \not\in \chainr$.
        Assume, to the contrary, that $\cprime \in \chainr$
        which implies that $\psi_c$ is a proper reduction of $\pi$.
        But now we have $\eta_c = \beta$, which is a contradiction,
        so $\psi_c$ is not a proper reduction of $\pi$,
        therefore $\cprime \not \in \chainr$.
        
        %
        Now assume that $\eta_c < \tau_c$.
        Let $\cprime = \Phi_{\chainb}(c) = c \cup \{ \ball{2413}{\eta_c} \}$.  
        We know by Lemma~\ref{lemma-2413-balloon-phi-gb}
        that this is a chain.
        Either $\kappa_c = \kappa_{\cprime}$,
        or $\kappa_{\cprime}$ 
        is a 2413-balloon.
        If $\kappa_c = \kappa_{\cprime}$,
        then $\cprime \not \in \chainr$.
        If $\kappa_{\cprime}$ is a 2413-balloon,
        then $\kappa_{\cprime} \not \in \redset$,
        so $\cprime \not \in \chainr$.
        Thus we must have $\cprime \not\in \chainr$.
        
        So now we have that 
        if $c \in \chainb$ 
        and $\eta_c = \tau_c$, 
        then $\Phi_{\chainb}(c)$ is a chain;
        and that for any $c \in \chainb$,
        $\Phi_{\chainb}(c) \in \chainb$.
        It follows that
        $\Phi_{\chainb}$ is a parity-reversing involution
        on $\chainb$.        
    \end{subproof}

    We now have that
    $\Phi_{\chaing}$ and
    $\Phi_{\chainb}$ 
    are parity-reversing involutions on
    $\chaing$ and
    $\chainb$ respectively.
    It follows from 
    Observation~\ref{observation-all-we-have-to-do}
    that
    $
    \mobp{\pi} = - \sum_{\sigma \in \redset} \mobp{\sigma}.
    $
    We now show how to express $\mobp{\sigma}$, where $\sigma \in \redset$,
    in terms of $\mobp{\beta}$.
    
    We start by noting that since $\beta$ is a 2413-balloon,
    then $\beta$ has no corners.
    Now,
    take the case where $\sigma = \ball{\ex{2}413}{\beta}$,
    which is the first permutation in 
    Figure~\ref{figure-2413-reductions}.
    Note that we can write 
    $\sigma = \oneminus ((\beta \minusone) \plusone)$.
    Applying Lemma~\ref{lemma-oneplus} to the 
    outermost three points in $\sigma$ 
    (those from the $\ex{2}413$), we find that
    $\mobp{\sigma} = -\mobp{\beta}$.
    The other cases are similar, and 
    this gives us:\footnote{This table is slightly redundant,
        as the entries are determined by the parity of the ``red'' points.
        We include it as later results have similar tables where some
        values of $\mobp{\sigma}$ are zero, and this gives a consistent presentation.}
    \begin{center}
        $
        \begin{array}{ccccc}
        \begin{array}{lr}
        \sigma & \mobp{\sigma} \\
        \midrule
        \ball{\ex{2}413}{\beta} & - \mobp{\beta} \\
        \ball{2\ex{4}13}{\beta} & - \mobp{\beta} \\
        \ball{24\ex{1}3}{\beta} & - \mobp{\beta} \\
        \ball{241\ex{3}}{\beta} & - \mobp{\beta} \\
        \phantom{x} & \phantom{x} \\		    
        \phantom{x} & \phantom{x} \\		    
        \end{array} 
        & \phantom{xxx} &
        \begin{array}{lr}
        \sigma & \mobp{\sigma} \\
        \midrule
        \ball{\ex{2}\ex{4}13}{\beta} & \mobp{\beta} \\
        \ball{\ex{2}4\ex{1}3}{\beta} & \mobp{\beta} \\
        \ball{\ex{2}41\ex{3}}{\beta} & \mobp{\beta} \\
        \ball{2\ex{4}\ex{1}3}{\beta} & \mobp{\beta} \\
        \ball{2\ex{4}1\ex{3}}{\beta} & \mobp{\beta} \\
        \ball{24\ex{1}\ex{3}}{\beta} & \mobp{\beta} \\
        \end{array} 
        & \phantom{xxx} &
        \begin{array}{lr}
        \sigma & \mobp{\sigma} \\
        \midrule
        \ball{2\ex{4}\ex{1}\ex{3}}{\beta} & - \mobp{\beta} \\
        \ball{\ex{2}4\ex{1}\ex{3}}{\beta} & - \mobp{\beta} \\
        \ball{\ex{2}\ex{4}1\ex{3}}{\beta} & - \mobp{\beta} \\
        \ball{\ex{2}\ex{4}\ex{1}3}{\beta} & - \mobp{\beta} \\
        \phantom{x} & \phantom{x} \\		    
        \phantom{x} & \phantom{x} \\		    
        \end{array} \\		  		    
        \end{array}
        $
    \end{center}
    It is now easy to see that 
    \[
    \sum_{\sigma \in \redset} \mobp{\sigma} = - 2 \mobp{\beta}
    \]
    and the result follows directly.	
\end{proof}

\section{The growth of the \mob function}
\label{section-growth-rate-of-mu}

We define
$\mobmax(n) = \max \{ \order{\mobp{\pi}} : \order{\pi} = n \}$.
Previous work in~\cite{Jelinek2020} and~\cite{Smith2013}
has shown that the growth of $\mobmax(n)$ is at least polynomial.  
We will show that the growth is at least exponential.
We have
\begin{theorem}
    \label{theorem-growth-of-mobius-function}
    For all $n$, 
    $\mobmax(n) \geq 2^{\lfloor n/4 \rfloor - 1 }$.
\end{theorem}    
\begin{proof}
    We start by a defining a function
    to construct a permutation of length $n$.
    \[
    \pi^{(n)} =
    \begin{cases}
    1                   & \text{If $n = 1$} \\
    12                  & \text{If $n = 2$} \\
    132                 & \text{If $n = 3$} \\
    2413                & \text{If $n = 4$} \\
    \ball{2413}{\pi^{(n-4)}} & \text{Otherwise}
    \end{cases}
    \]
    Note that for $n > 8$, $\pi^{(n)}$ is a double 2413-balloon.
    It is simple to calculate $\mobp{\pi^{(n)}}$ for $n=1, \ldots, 8$,
    and these values are given below.
    
    \begin{center}
    \begin{tabular}{lcl}
    $\mobp{\pi^{(1)}} = \mobp{1} = 1$,  &
    \phantom{xxxxx} & 
    $\mobp{\pi^{(5)}} = \mobp{25314} =  4$, \\
    $\mobp{\pi^{(2)}} = \mobp{12} = -1$, && 
    $\mobp{\pi^{(6)}} = \mobp{263415} = -1$, \\
    $\mobp{\pi^{(3)}} = \mobp{132} = 1$, &&
    $\mobp{\pi^{(7)}} = \mobp{2735416} = 1$, \\
    $\mobp{\pi^{(4)}} = \mobp{2413} = -3$, && 
    $\mobp{\pi^{(8)}} = \mobp{28463517}  = -6$.
    \end{tabular}    
    \end{center}
    These values match
    Theorem~\ref{theorem-growth-of-mobius-function},
    and so this is true for $n \leq 8$.
    For $n > 8$, 
    $\mobp{\pi^{(n)}} = 2 \mobp{\pi^{(n-4)}}$ 
    by Theorem~\ref{theorem-2413-balloon-beta-is-a-balloon},
    and the result follows immediately.    
\end{proof}
\begin{remark}
    \label{remark-simples-in-2413-balloons}
    It is easy to see that, with the definitions above,
    the only simple permutations that can be contained
    in $\pi^{(n)}$ are
    $1$, 
    $12$, 
    $21$,
    $2413$, and
    $25314$.
    This answers
    Problem 4.4 in~\cite{Jelinek2020},
    which asks
    whether $\mobp{\pi}$ is bounded on a hereditary 
    class which contains only finitely many simple permutations,
    as, by Theorem~\ref{theorem-growth-of-mobius-function},
    we have unbounded growth, but only finitely many simple 
    permutations.
\end{remark}

If we repeat the ballooning process, as we do in $\pi^{(n)}$,
then the permutation plot is rather striking.  
We illustrate this in 
Figure~\ref{figure-multiple-2413-ballons},
which shows $\pi^{(21)}$.
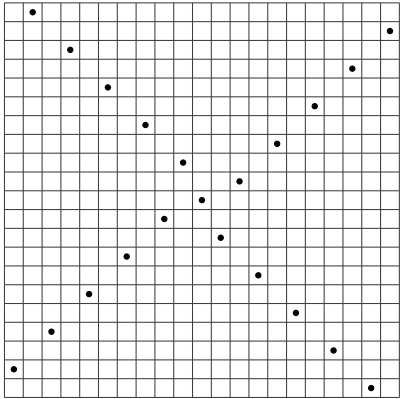
\begin{figure}[ht!]
    \begin{center}
        \begin{tikzpicture}[scale=0.25]
        \sgrid{21}
        \point{01}{02}   
        \point{03}{04}   
        \point{05}{06}   
        \point{07}{08}   
        \point{09}{10} 
        \point{02}{21} 
        \point{04}{19} 
        \point{06}{17} 
        \point{08}{15} 
        \point{10}{13} 
        \point{20}{01} 
        \point{18}{03} 
        \point{16}{05} 
        \point{14}{07} 
        \point{12}{09} 
        \point{21}{20} 
        \point{19}{18} 
        \point{17}{16} 
        \point{15}{14} 
        \point{13}{12} 
        \point{11}{11}
        \end{tikzpicture}
        \caption{A permutation plot showing $\pi^{(21)}$.}
        \label{figure-multiple-2413-ballons}
    \end{center}
\end{figure}

\section{The \mob function of 2413-balloons}
\label{section-2413-balloons}

Theorem~\ref{theorem-2413-balloon-beta-is-a-balloon} 
gives us an expression for the value of the \mob function
$\mobp{\pi}$ when $\pi$ is a double 2413-balloon.
We expand on this to find an expression for 
the \mob function $\mobp{\pi}$ when
$\pi$ is any 2413-balloon.

We start with a Lemma
that handles the case where
$\beta$ is not a 2413-balloon,
and has more than four points.
The structure of our proof is similar to that of 
Theorem~\ref{theorem-2413-balloon-beta-is-a-balloon},
but we present a complete
argument to aid readability.

We will show
\begin{lemma}
    \label{lemma-2413-balloon-beta-not-a-balloon}
    Let $\pi = \ball{2413}{\beta}$,
    where $\beta$ is not a 2413-balloon,
    and $\order{\beta} > 4$.
    Then $\mobp{\pi} = \mobp{\beta}$.	
\end{lemma}
\begin{proof}
    First note that if $\beta$ is monotonic, then
    by Corollary~\ref{corollary-monotonic-interval}
    we have $\mobp{\beta} = 0 = \mobp{\pi}$.
    For the remainder of this proof, we assume that
    $\beta$ is not monotonic.
    
    If $\beta$ has one corner, then
    without loss of generality, we can assume by symmetry that
    $\beta = \oneplus \gamma$.
    Similarly, if $\beta$ has two corners,
    then we can assume that $\beta = \oneplus \gamma \plusone$.
    
    As before, we will use
    Observation~\ref{observation-all-we-have-to-do}.
    We will show that
    $\Phi_{\chainb}$ is a parity-reversing involution on $\chainb$.
    Once we have shown that we have parity-reversing involutions,
    we will then show how to express the Hall sum of $\chainr$
    in terms of $\mobp{\beta}$.    
    
    The proper reductions of $\pi$ depend on the number of corners of $\beta$.
    Below we list the improper reductions of $\pi$ for each case.
    
    \begin{center}    
        \begin{tabular}{lr}
            \toprule
            Corners in $\beta$ & Improper reductions of $\pi$  \\
            \midrule
            No corners & 
            None. \\
            One corner ($\beta = 1 \oplus \gamma$) &
            $\ball{\ex{2}\ex{4}13}{\beta}$,
            $\ball{\ex{2}\ex{4}1\ex{3}}{\beta}$,
            $\ball{\ex{2}\ex{4}\ex{1}3}{\beta}$,
            and
            $\beta$. \\
            Two corners ($\beta = 1 \oplus \gamma  \oplus 1$) &
            $\ball{\ex{2}\ex{4}13}{\beta}$,
            $\ball{24\ex{1}\ex{3}}{\beta}$,
            \\ & 
            $\ball{2\ex{4}\ex{1}\ex{3}}{\beta}$,
            $\ball{\ex{2}4\ex{1}\ex{3}}{\beta}$,
            $\ball{\ex{2}\ex{4}1\ex{3}}{\beta}$,
            $\ball{\ex{2}\ex{4}\ex{1}3}{\beta}$, 
            \\ & and    
            $\beta$. \\
            \bottomrule
        \end{tabular}        
    \end{center}
    
    %
    %
    \begin{subproof}[Proof that $\Phi_{\chainb}$ is a parity-reversing involution on $\chainb$.]
        Let $c$ be a chain in $\chainb$.
        
        First, assume that $\eta_c = \tau_c$, so
        $\cprime = c \setminus \{ \ball{2413}{\tau_c} \}$.
        We start by showing that $\cprime$ is a valid chain.
        Assume otherwise, which implies $\tau_c = \beta$.
        
        If $\beta$ has no corners, 
        then
        $\psi_c \in \redset$, so $c \in \chainr$,
        which is a contradiction.
        
        If $\beta$ has one corner,
        then
        either $\psi_c \in \redset$, 
        which is a contradiction,
        or $\psi_c \not \in \redset$.
        In the latter case,
        assume, without loss of generality,
        that $\beta = 1 \oplus \gamma$.
        Then 
        $\ball{\ex{2}\ex{4}13}{\beta}       = \ball{2\ex{4}13}{\gamma}$,
        $\ball{\ex{2}\ex{4}1\ex{3}}{\beta}  = \ball{2\ex{4}1\ex{3}}{\gamma}$,
        $\ball{\ex{2}\ex{4}\ex{1}3}{\beta}  = \ball{2\ex{4}\ex{1}3}{\gamma}$,
        and
        $\beta                              = \ball{2\ex{4}\ex{1}\ex{3}}{\gamma}$.
        Thus in all cases where 
        $\psi_c \not \in \redset$,
        we have that $\eta_c$ is not minimal, 
        which is a contradiction.
        
        Finally, if $\beta$ has two corners,
        then
        either $\psi_c \in \redset$,
        which is a contradiction,
        or $\psi_c \not \in \redset$.
        The latter case implies that
        $\psi_c = \beta$,
        and then we have that
        either 
        $\psi_c = 1 \oplus \gamma \oplus 1 = \ball{2\ex{4}\ex{1}3}{\gamma}$
        or
        $\psi_c = 1 \ominus \gamma \ominus 1 = \ball{\ex{2}41\ex{3}}{\gamma}$,
        so
        $\eta_c$ is not minimal,
        which is a contradiction.
        
        Thus we have that $\cprime$ must be a chain,
        and, moreover, $\tau_c \neq \beta$.
        
        We now show that $\cprime \not\in \chainr$.
        Assume, to the contrary, that $\cprime \in \chainr$
        which implies that $\psi_c$ is a proper reduction of $\pi$.
        But now we have $\eta_c = \beta$, 
        but this would give $\tau_c = \beta$,
        which is a contradiction,
        therefore $\cprime \not \in \chainr$.
                
        Now assume that $\eta_c < \tau_c$.
        Let $\cprime = \Phi_{\chainb}(c) = c \cup \{ \ball{2413}{\eta_c} \}$,
        and we know from Lemma~\ref{lemma-2413-balloon-phi-gb}
        that $\cprime$ is a chain.
        Now either $\kappa_c = \kappa_{\cprime}$,
        or $\kappa_{\cprime}=\ball{2413}{\eta_c}$ is a 2413-balloon.
        In either case we have $\cprime \not \in \chainr$.
        
        So if $c \in \chainb$, then $\Phi_{\chainb}(c)$
        is a chain in $\chainb$,
        and thus $\Phi_{\chainb}$ is a parity-reversing involution.
    \end{subproof}    
    
    We have shown that
    $\Phi_{\chaing}$ and
    $\Phi_{\chainb}$ 
    are parity-reversing involutions on
    $\chaing$ and
    $\chainb$ respectively.
    It follows from 
    Observation~\ref{observation-all-we-have-to-do}
    that
    $
    \mobp{\pi} = - \sum_{\sigma \in \redset} \mobp{\sigma}.
    $
    We now show how to express $\mobp{\sigma}$, where $\sigma \in \redset$
    in terms of $\mobp{\beta}$.
    We use a similar mechanism to that used
    in Theorem~\ref{theorem-2413-balloon-beta-is-a-balloon}.
    There are some additional considerations
    where $\beta$ has one or two corners.
    
    As an example, 
    take the case where $\sigma = \ball{2\ex{4}13}{\beta}$,
    and $\beta$ has one corner, and so, by our assumption,  can be
    written as $\oneplus \gamma$.
    We can write 
    $\sigma = ((\oneplus \beta) \minusone) \plusone$,
    and expanding $\beta$ we have
    $\sigma = ((\oneplus \oneplus \gamma) \minusone) \plusone$,    
    Applying Lemma~\ref{lemma-oneplus} to the 
    outermost two points in $\sigma$, 
    we find that
    $\mobp{\sigma} = \mobp{\oneplus \oneplus \gamma}$,
    and by 
    Lemma~\ref{lemma-oneplus-oneplus} we now have
    $\mobp{\sigma} = 0$.
    Because of this, our analysis depends on the number of corners of $\beta$,
    and we consider each case separately below.
    
    If $\beta$ has no corners, then we have
    \begin{center}
        $
        \begin{array}{ccccc}
        \begin{array}{lr}
        \sigma & \mobp{\sigma} \\
        \midrule
        \ball{\ex{2}413}{\beta} & - \mobp{\beta} \\
        \ball{2\ex{4}13}{\beta} & - \mobp{\beta} \\
        \ball{24\ex{1}3}{\beta} & - \mobp{\beta} \\
        \ball{241\ex{3}}{\beta} & - \mobp{\beta} \\
        \phantom{x} & \phantom{x} \\		    
        \phantom{x} & \phantom{x} \\		    
        \end{array} 
        & \phantom{xxx} &
        \begin{array}{lr}
        \sigma & \mobp{\sigma} \\
        \midrule
        \ball{\ex{2}\ex{4}13}{\beta} & \mobp{\beta} \\
        \ball{\ex{2}4\ex{1}3}{\beta} & \mobp{\beta} \\
        \ball{\ex{2}41\ex{3}}{\beta} & \mobp{\beta} \\
        \ball{2\ex{4}\ex{1}3}{\beta} & \mobp{\beta} \\
        \ball{2\ex{4}1\ex{3}}{\beta} & \mobp{\beta} \\
        \ball{24\ex{1}\ex{3}}{\beta} & \mobp{\beta} \\
        \end{array} 
        & \phantom{xxx} &
        \begin{array}{lr}
        \sigma & \mobp{\sigma} \\
        \midrule
        \ball{2\ex{4}\ex{1}\ex{3}}{\beta} & - \mobp{\beta} \\
        \ball{\ex{2}4\ex{1}\ex{3}}{\beta} & - \mobp{\beta} \\
        \ball{\ex{2}\ex{4}1\ex{3}}{\beta} & - \mobp{\beta} \\
        \ball{\ex{2}\ex{4}\ex{1}3}{\beta} & - \mobp{\beta} \\
        \phantom{x} & \phantom{x} \\		    
        \beta & \mobp{\beta} \\		    
        \end{array} \\		  		    
        \end{array}
        $
    \end{center}
    
    If $\beta$ has one corner, under our assumption
    that $\beta$ = $\oneplus \gamma$, we have
    \begin{center}
        $
        \begin{array}{ccccc}
        \begin{array}{lr}
        \sigma & \mobp{\sigma} \\
        \midrule
        \ball{\ex{2}413}{\beta} & - \mobp{\beta} \\
        \ball{2\ex{4}13}{\beta} & 0 \\
        \ball{24\ex{1}3}{\beta} & - \mobp{\beta} \\
        \ball{241\ex{3}}{\beta} & - \mobp{\beta} \\
        \phantom{x} & \phantom{x} \\		    
        \end{array} 
        & \phantom{xxx} &
        \begin{array}{lr}
        \sigma & \mobp{\sigma} \\
        \midrule
        \ball{\ex{2}4\ex{1}3}{\beta} & \mobp{\beta} \\
        \ball{\ex{2}41\ex{3}}{\beta} & \mobp{\beta} \\
        \ball{2\ex{4}\ex{1}3}{\beta} & 0 \\
        \ball{2\ex{4}1\ex{3}}{\beta} & 0 \\
        \ball{24\ex{1}\ex{3}}{\beta} & \mobp{\beta} \\
        \end{array} 
        & \phantom{xxx} &
        \begin{array}{lr}
        \sigma & \mobp{\sigma} \\
        \midrule
        \ball{2\ex{4}\ex{1}\ex{3}}{\beta} & 0 \\
        \ball{\ex{2}4\ex{1}\ex{3}}{\beta} & - \mobp{\beta} \\
        \phantom{x} & \phantom{x} \\ 
        \phantom{x} & \phantom{x} \\
        \phantom{x} & \phantom{x} \\		    
        \end{array} \\		  		    
        \end{array}
        $
    \end{center}
    
    Finally, if $\beta$ has two corners, under
    our assumption that $\beta = \oneplus \gamma \plusone$,
    we have
    \begin{center}
        $
        \begin{array}{ccc}
        \begin{array}{lr}
        \sigma & \mobp{\sigma} \\
        \midrule
        \ball{\ex{2}413}{\beta} & - \mobp{\beta} \\
        \ball{2\ex{4}13}{\beta} & 0 \\
        \ball{24\ex{1}3}{\beta} & 0 \\
        \ball{241\ex{3}}{\beta} & - \mobp{\beta} \\
        \end{array} 
        & \phantom{xxx} &
        \begin{array}{lr}
        \sigma & \mobp{\sigma} \\
        \midrule
        \ball{\ex{2}4\ex{1}3}{\beta} & 0 \\
        \ball{\ex{2}41\ex{3}}{\beta} & \mobp{\beta} \\
        \ball{2\ex{4}\ex{1}3}{\beta} & 0 \\
        \ball{2\ex{4}1\ex{3}}{\beta} & 0 \\
        \end{array} 
        \end{array}
        $
    \end{center}
    In all three cases we have    
    \[
    \sum_{\sigma \in \redset} \mobp{\sigma} = - \mobp{\beta}
    \]
    and the result follows directly.
\end{proof}	

We are now in a position to state and prove the main Theorem
for this section.
\begin{theorem}
    \label{theorem-2413-balloons}
    Let $\pi = \ball{2413}{\beta}$.  Then
    \begin{align*}
    \mobp{\pi} = 
    \begin{cases}
    4 & \text{If $\beta = 1$} \\
    -6 & \text{If $\beta = 2413$} \\
    2 \mobp{\beta} & \text{If $\beta$ is a 2413-balloon} \\
    \mobp{\beta}  & \text{Otherwise}.	
    \end{cases}
    \end{align*}
\end{theorem}
\begin{proof}
    The value of $\mobp{\ball{2413}{\beta}}$
    for the symmetry classes of $\beta$ 
    with $\order{\beta} \leq 4$ are shown below.
    \begin{center}
        $
        \begin{array}{ccc}
        \begin{array}{lrr}
        \beta & \mobp{\beta} &\mobp{\ball{2413}{\beta}} \\
        \midrule
        1    &  1 &  4 \\
        12   & -1 & -1 \\
        123  &  0 &  0 \\
        132  &  1 &  1 \\
        1234 &  0 &  0 \\
        1243 &  0 &  0 \\
        \end{array} 
        & \phantom{xxx} &
        \begin{array}{lrr}
        \beta & \mobp{\beta} &\mobp{\ball{2413}{\beta}} \\
        \midrule
        1324 & -1 & -1 \\
        1342 & -1 & -1 \\
        1432 &  0 &  0 \\
        2143 & -1 & -1 \\
        2413 & -3 & -6 \\
        \phantom{x} & \phantom{x} \\		    
        \end{array} \\		  		    
        \end{array}
        $
    \end{center}        
    It is easy to see that
    these values meet Theorem~\ref{theorem-2413-balloons}.
    We now combine
    Theorem~\ref{theorem-2413-balloon-beta-is-a-balloon} and
    Lemma~\ref{lemma-2413-balloon-beta-not-a-balloon}
    to complete the proof.
\end{proof}

\section{Concluding remarks}
\label{section-concluding-remarks}

\subsection{Generalising the balloon operator}

Given two permutations $\alpha$ and $\beta$,
with lengths $a$ and $b$ respectively,
and two integers $i,j$ which satisfy
$0 \leq i,j \leq a$, the
\emph{$i,j$-balloon} of $\beta$ by $\alpha$,
written as 
$\ballgen{i,j}{\alpha}{\beta}$,
is the 
permutation formed by inserting
the permutation $\beta$ into $\alpha$
between the $i$-th and $i+1$-th columns of $\alpha$,
and 
between the $j$-th and $j+1$-th rows of $\alpha$.
The integers $i$ and $j$ are, collectively, the
\emph{indexes} of the balloon.

Formally, we have
\begin{align*}
(\ballgen{i,j}{\alpha}{\beta})_x
& =
\begin{cases}
\alpha_x 
& \text{if $x \leq i$ and $\alpha_x \leq j $}\\
\alpha_x + \order{\beta} 
& \text{if $x \leq i$ and $\alpha_x > j $}\\
\beta_{x-i} + j 
& \text{if $x > i $ and $x \leq i+\order{\beta}$ }\\
\alpha_{x-\order{\beta}} 
& \text{if $x > i+\order{\beta}$ and $\alpha_{x-\order{\beta}} \leq j $}\\
\alpha_{x-\order{\beta}} + \order{\beta} 
& \text{if $x > i+\order{\beta}$ and $\alpha_{x-\order{\beta}} > j $}\\ 
\end{cases}
\end{align*}
As before, the balloon notation is not associative.
Unlike 2413-balloons, which have to be interpreted
as right-associative, 
generalized balloons can use brackets to define associativity.
Note that the 2413-balloon defined in Section~\ref{section-definitions-and-notation}
are written as $\ballgen{2,2}{2413}{\beta}$ in our
generalized notation.

We remark that for any $\alpha$ and any $\beta$, we have
$\ballgen{0,0}{\alpha}{\beta} = \alpha \oplus \beta$,
and we can easily determine $\mobp{\alpha \oplus \beta}$
using results from
Propositions~1~and~2 of
Burstein, 
Jel{\'{i}}nek, 
Jel{\'{i}}nkov{\'{a}} and 
Steingr{\'{i}}msson~\cite{Burstein2011}.

\subsection{Generalised 2413-balloons}

If we restrict $\alpha$ to 2413,
then, up to symmetry, there are seven 
possible values for the indexes:
$(0,0)$,
$(0,1)$,
$(0,2)$,
$(1,0)$,
$(1,1)$,
$(1,2)$, and
$(2,2)$.
Theorem~\ref{theorem-2413-balloons} 
handles the case where the indexes are $(2,2)$,
and~\cite{Burstein2011} 
handles the case where the indexes are $(0,0)$.
For the other indexes, we have
\begin{conjecture}
    Let $\pi = \ballgen{i,j}{2413}{\beta}$,
    where 
    $(i,j) \in 
    \{
    (0,1), (0,2), (1,1), (1,2)
    \}$.
    Then
    \[
    \mobp{\pi} =
    \begin{cases}
    0 & \text{If $(i,j) = (0,1)$ and $\beta = \tau \plusone$} \\
    0 & \text{If $(i,j) = (0,2)$ and $\beta = \tau \minusone$} \\
    0 & \text{If $(i,j) = (1,1)$ and $\beta = \oneminus \tau$ or 12} \\
    0 & \text{If $(i,j) = (1,2)$ and $\beta = \oneplus \tau$} \\
    \mobp{\beta} & \text{Otherwise.}
    \end{cases}
    \]
\end{conjecture}
and
\begin{conjecture}
    \label{conjecture-1-0-2413-balloons}
    Let $\pi = \ballgen{1,0}{2413}{\beta}$.
    Then
    \[
    \mobp{\pi} =
    \begin{cases}
    6 & \text{If $\beta = 1$} \\
    -2 & \text{If $\beta = 21$} \\
    0 & \text{If $\beta = 312$} \\
    2 \mobp{\beta} & \text{If $\beta = \ballgen{1,0}{2413}{\gamma}$} \\
    \mobp{\beta} & \text{Otherwise.}
    \end{cases}
    \]
\end{conjecture}
We remark here that
Theorem~\ref{theorem-2413-balloons}
and
Conjecture~\ref{conjecture-1-0-2413-balloons}
have a very similar structure.  
It is not clear to us whether this similarity
is coincidental, or whether there is 
some deeper reason.

\subsection{Bounding the \mob function on hereditary classes}

Corollary 24 in 
Burstein, Jel{\'{i}}nek, Jel{\'{i}}nkov{\'{a}} 
and Steingr{\'{i}}msson~\cite{Burstein2011}
gives us that if $\pi$ is separable,
then $\mobp{\pi} \in \{ 0, \pm1 \}$.  
The simple permutations in the 
hereditary class of separable permutations
are 
$1$,
$12$, and
$21$.
In Remark~\ref{remark-simples-in-2413-balloons}
we have unbounded growth where the 
simple permutations in the hereditary class are just
$1$, 
$12$, 
$21$,
$2413$, and
$25314$,
so adding $2413$ and $25314$ to the simple permutations
moves us from bounded growth
to unbounded growth.
This then leads to:
\begin{question}
    If $C$ is a hereditary class containing just the simples
    $1$, 
    $12$, 
    $21$ and
    $2413$,
    and $\pi \in C$, then
    is $\mobp{\pi}$ bounded?
    Further, if $D$ is a hereditary class containing just the simples
    $1$, 
    $12$, 
    $21$,
    $2413$, and
    $3142$,
    and $\pi \in D$, then
    is $\mobp{\pi}$ bounded?    
\end{question}

\paragraph*{Acknowledgements.}

I would like to thank my supervisor, Robert Brignall, 
for his patience and help while I was writing this paper,
and 
Einar Steingr{\'{i}}msson 
and 
Jan Kyn{\v{c}}l
for their comments.
I would also like to thank Jan 
for spotting an error in an earlier version
of what is now Lemma~\ref{lemma-2413-balloon-phi-gb},
and an anonymous reviewer
for identifying an error in Lemma~\ref{lemma-2413-balloon-beta-not-a-balloon},
and suggesting a solution.


\end{document}